\documentclass[a4paper,12pt,twoside]{article}
\date{October 30th, 2017}

\usepackage[english]{babel}
\usepackage{lmodern,fancyhdr,lastpage,nccfoots,caption}
\usepackage{hyperref}
\hypersetup{
    pdfstartview={FitH},     
    pdftitle={Stabilization of the Homotopy Groups of the Moduli Spaces of k-Higgs Bundles},             
    pdfauthor={Ronald Alberto Z\'u\~niga-Rojas},     
    pdfsubject={Vector bundles on curves and their moduli},           
    pdfkeywords={Moduli of Higgs Bundles, Variations of Hodge Structures, Vector Bundles},           
    pdfnewwindow=true,        
    colorlinks=true,          
    linkcolor=blue,           
    citecolor=red,            
    urlcolor=blue             
}

\usepackage[utf8]{inputenc}
\usepackage[all]{xy}
\usepackage{array}
\usepackage{graphicx,color}
\usepackage{amsmath,amssymb,amsthm}
\usepackage{enumerate}
\usepackage[a4paper,margin=1in]{geometry}
\usepackage{textcomp}
\usepackage{mathrsfs}
\usepackage{hyphenat}
\fancypagestyle{plain}{
\fancyhf{}}

\DeclareMathOperator{\End}{End}     
\DeclareMathOperator{\GCD}{GCD}     
\DeclareMathOperator{\Hom}{Hom}     
\DeclareMathOperator{\id}{id}       
\DeclareMathOperator{\Jac}{Jac}     
\DeclareMathOperator{\rk}{rk}       
\DeclareMathOperator{\Sym}{Sym}     
\DeclareMathOperator{\tr}{tr}       

\newcommand{\la}{\lambda}           
\newcommand{\Om}{\varOmega}         
\newcommand{\sg}{\sigma}            

\newcommand{\bC}{\mathbb{C}}        
\newcommand{\bN}{\mathbb{N}}        
\newcommand{\bR}{\mathbb{R}}        
\newcommand{\bS}{\mathbb{S}}        
\newcommand{\bZ}{\mathbb{Z}}        
\newcommand{\gG}{\mathcal{G}}       
\newcommand{\sJ}{\mathcal{J}}	    
\newcommand{\sM}{\mathcal{M}}       
\newcommand{\sN}{\mathcal{N}}       
\newcommand{\sO}{\mathcal{O}}       

\newcommand{\dd}{\mathbf{d}}        
\newcommand{\rr}{\mathbf{r}}        

\renewcommand{\geq}{\geqslant}      
\renewcommand{\leq}{\leqslant}      
\newcommand{\ox}{\otimes}           
\renewcommand{\:}{\colon}           

\def\Circlearrowright{\ensuremath{%
  \rotatebox[origin=c]{180}{$\circlearrowright$}}}

\newcommand{\word}[1]{\quad\text{#1}\quad} 

\theoremstyle{plain}
\newtheorem{Th}{Theorem}[section]   
\newtheorem*{nonum-Th}{Theorem}     
\newtheorem{Prop}[Th]{Proposition}  
\newtheorem{Lem}[Th]{Lemma}         
\newtheorem{Cor}[Th]{Corollary}     
\newtheorem*{nonum-Cor}{Corollary}  

\theoremstyle{definition}
\newtheorem{Def}[Th]{Definition}    
\newtheorem{assumption}[Th]{Assumption}  

\theoremstyle{remark}
\newtheorem{Rmk}[Th]{Remark}        

\numberwithin{equation}{section}

\DeclareRobustCommand{\QEDA}{\ifmmode
  \else \leavevmode\unskip\penalty9999 \hbox{}\nobreak\hfill \fi
  \quad\hbox{\qedasymbol}}
\newcommand{\qedasymbol}{$\boxminus$} 

\makeatletter

\renewcommand{\section}{\@startsection{section}{1}{\z@}%
                        {-3.5ex \@plus -1ex \@minus -.2ex}%
                        {2.3ex \@plus.2ex}%
                        {\normalfont\large\bfseries}}
\renewcommand{\subsection}{\@startsection{subsection}{2}{\z@}%
                        {-3.25ex \@plus -1ex \@minus -.2ex}%
                        {1.5ex \@plus .2ex}%
                        {\normalfont\normalsize\bfseries}}
\renewcommand{\subsubsection}{\@startsection{subsubsection}{3}{\z@}%
                        {-3.25ex \@plus -1ex \@minus -.2ex}%
                        {1.5ex \@plus .2ex}%
                        {\normalfont\normalsize\itshape}}
\renewcommand{\@dotsep}{200} 
\makeatother

\begin{document}

\thispagestyle{empty}

\begin{center}
\Large
\textsc{Stabilization of the Homotopy Groups of\\
the Moduli Spaces of $k$-Higgs Bundles}\\

\bigskip
\normalsize
October 30th, 2017\\
  
\bigskip
\emph{Ronald A. Z\'u\~niga-Rojas}\footnote{\scriptsize Supported by Universidad de Costa Rica through CIMM
  (Centro de Investigaciones Matem\'aticas y Metamatem\'aticas), 
  Project 820-B5-202. This work is based on the Ph.D. Project \cite{z-r} called
  ``Homotopy Groups of the Moduli Space of Higgs Bundles'', supported
  by FEDER through Programa Operacional Factores de Competitividade-COMPETE,
  and also supported by FCT (Funda\c{c}\~ao para a Ci\^encia e a Tecnologia)
  through the projects PTDC/MAT-GEO/0675/2012 and PEst-C/MAT/UI0144/2013 
  with grant reference SFRH/BD/51174/2010.}\\[6pt]
\small Centro de Investigaciones Matem\'aticas y Metamatem\'aticas CIMM\\
\small Universidad de Costa Rica UCR\\
\small San Jos\'e 11501, Costa Rica\\
\small e-mail: \texttt{ronald.zunigarojas@ucr.ac.cr}
\end{center}

\noindent
\textbf{Abstract.}
The work of Hausel proves that the Bia\l{}ynicki-Birula stratification 
of the moduli space of rank two Higgs bundles coincides with its 
Shatz stratification. He uses that to estimate some homotopy groups 
of the moduli spaces of $k$-Higgs bundles of rank two. Unfortunately, 
those two stratifications do not coincide in general. Here, the 
objective is to present a different proof of the stabilization of the 
homotopy groups of $\sM^k(2,d)$, and generalize it to $\sM^k(3,d)$, 
the moduli spaces of $k$-Higgs bundles of degree~$d$, and ranks two 
and three respectively, over a compact Riemann surface $X$, using the 
results from the works of Hausel and Thaddeus, among other tools.

\medskip

\begin{flushleft}

\noindent
\textbf{AMS 2010 MSC classes:} Primary \texttt{55Q52}; Secondaries \texttt{14H60}, \texttt{14D07}.
\end{flushleft}

\section*{Introduction}
\addcontentsline{toc}{section}{Introduction}
\label{sec:0} 

\quad In this work, we estimate some homotopy groups of the moduli spaces of
$k$-Higgs bundles $\sM^k(r,d)$\ over a compact Riemann surface $X$ of
genus $g > 2$. This space was first introduced by
Hitchin~\cite{hit2}; and then, it was worked by Hausel~\cite{hau},
where he estimated some of the homotopy groups working the particular
case of rank two, and denoting 
$\displaystyle
\sM^{\infty} = \lim_{k\to \infty}\sM^k
$ 
as the direct limit of the sequence. 

The co-prime condition 
$\GCD(r,d) = 1$ implies that $\sM^k(r,d)$ 
is smooth. We shall do the estimate with Higgs bundles of 
fixed determinant $\det(E) = \Lambda \in \sJ^d$, where $\sJ^d$ 
is the Jacobian of degree $d$ line bundles on $X$, to ensure 
that $\sN(r,d)$ and $\sM(r,d)$ are simply connected. Denote 
$\sM^k_{\Lambda}$ as the moduli space of $k$-Higgs bundles with 
determinant $\Lambda$, and 
$\displaystyle
\sM^{\infty}_{\Lambda} = \lim_{k\to \infty}\sM^k_{\Lambda}
$
as the direct limit of these moduli spaces, as before. Hence, the group action 
$\pi_1(\sM^k_\Lambda) \Circlearrowright \pi_n(\sM^\infty_\Lambda,\sM^k_\Lambda)$ 
will be trivial.

Hausel \cite{hau} estimates the homotopy groups $\pi_n(\sM^k(2,1))$
using two main tools: first the coincidence mentioned before between
the Bia\l{}ynicki-Birula stratification and the Shatz stratification; and
second, the well-behaved embeddings
$\sM^k(2,1) \hookrightarrow \sM^{k+1}(2,1)$. These inclusions are also
well-behaved in general for $\GCD(r,d) = 1$; nevertheless, those two
stratifications above mentioned do not coincide in general (see for
instance~\cite{gzr}).

In this paper, our estimate is based on the embeddings
$$
\sM^k(r,d) \hookrightarrow \sM^{k+1}(r,d)
$$ 
and their good behavior,
notwithstanding the non-coincidence between stratifications when the 
rank is $r = 3$. The paper is organized as follows: in section~\ref{sec:1} 
we recall some facts about vector bundles and Higgs bundles; in 
section~\ref{sec:2}, we present the cohomology ring $H^n(\sM^k)$; in 
section~\ref{sec:3}, we discuss the most relevant results about the cohomology 
and the homotopy of the moduli spaces $\sM^k$; finally, in section~\ref{sec:4}, 
subsection~\ref{ssec:4.1} we estimate the homotopy groups 
of $\sM^k$ under the assumption that $\pi_1(\sM^k)$ acts trivially on 
$\pi_n(\sM^\infty,\sM^k)$, and hence, in subsection~\ref{ssec:4.2} we
present and prove the main result:

\begin{nonum-Th}
(Corollary \ref{[Z-R]MainResult2FixedDet})
Suppose the rank is either $r = 2$ or $r = 3$, and $\GCD(r,d) = 1$. 
Then, for all $n$ exists $k_0$, depending on $n$, such that 
 $$
 \pi_j\left(\sM_{\Lambda}^{k}(r,d)\right) \xrightarrow{\quad \cong \quad} \pi_j\left(\sM_{\Lambda}^\infty(r,d)\right)
 $$
 for all $k \geq k_0$ and for all $j \leq n-1$.
\end{nonum-Th}

\section{Preliminary definitions} 
\label{sec:1}

Let $X$ be a compact Riemann surface of genus $g > 2$ and let
$K = T^*X$ be the canonical line bundle of $X$. Note that,
algebraically, $X$ is also a nonsingular complex projective algebraic
curve.

\begin{Def} 
 A \emph{Higgs bundle} over $X$ is a pair $(E, \Phi)$ where 
 $E \to X$ is a holomorphic vector bundle and $\Phi\: E \to E \ox K$
 is an endomorphism of $E$ twisted by~$K$, which is called a 
 \emph{Higgs field}. Note that $\Phi \in H^0(X;\End(E) \ox K)$.
\end{Def}

\begin{Def} 
 For a vector bundle $E \to X$, we denote the \emph{rank} of~$E$ by
 $\rk(E) = r$ and the \emph{degree} of $E$ by $\deg(E) = d$. Then,
 for any smooth bundle $E \to X$ the \emph{slope} is defined to be
  \begin{equation}
    \mu(E) := \frac{\deg(E)}{\rk(E)} = \frac{d}{r}.
    \label{slope} 
  \end{equation} 
 A vector bundle $E \to X$ is called \emph{semistable} if 
 $\mu(F) \leq \mu(E)$ for any $F$ such that
 $0 \subsetneq F \subseteq E$. Similarly, a vector bundle $E \to X$ is
 called \emph{stable} if $\mu(F) < \mu(E)$ for any nonzero proper
 subbundle $0 \subsetneq F \subsetneq E$. Finally, $E$ is called 
 \emph{polystable} if it is the direct sum of stable subbundles,
 all of the same slope.
\end{Def}

\begin{Def} 
 A subbundle $F \subset E$ is said to be \emph{$\Phi$-invariant} 
 if $\Phi(F) \subset F \ox K$. A Higgs bundle is said to be 
 \emph{semistable} [respectively, \emph{stable}] if 
 $\mu(F) \leq \mu(E)$ [resp., $\mu(F) < \mu(E)$] for 
 any nonzero $\Phi$-invariant subbundle $F \subseteq E$
 [resp., $F \subsetneq E$]. Finally, $(E,\Phi)$ is called
 \emph{polystable} if it is the direct sum of stable
 $\Phi$-invariant subbundles, all of the same slope.
\end{Def}

Fixing the rank $\rk(E) = r$ and the degree $\deg(E) = d$ of a Higgs
bundle $(E,\Phi)$, the isomorphism classes of polystable bundles are
parametrized by a quasi-projective variety: the moduli space
$\sM(r,d)$. Constructions of this space can be found in the work of
Hitchin~\cite{hit2}, using gauge theory, or in the work of
Nitsure~\cite{nit}, using algebraic geometry methods.

An important feature of $\sM(r,d)$ is that it carries an action
of~$\bC^*$: $z \cdot (E, \Phi) = (E, z \cdot \Phi)$. According to
Hitchin~\cite{hit2}, $(\sM,I,\Om)$ is a K\"ahler manifold, where
$I$ is its complex structure and $\Om$ its corresponding
K\"ahler form. Furthermore, $\bC^*$ acts on $\sM$ biholomorphically
with respect to the complex structure $I$ by the action mentioned
above, where the K\"ahler form $\Om$ is invariant under the
induced action
$e^{i\theta} \cdot (E, \Phi) = (E, e^{i\theta} \cdot \Phi)$ of the
circle $\bS^1 \subset \bC^*$. Besides, this circle action is
Hamiltonian, with proper momentum map
$f \: \sM \to \bR$ defined by:

\begin{equation}
f(E, \Phi) = \frac{1}{2\pi} \|\Phi\|_{L^2}^2
= \frac{i}{2\pi} \int_X \tr(\Phi \Phi^*),
\label{momentum_map_intro} 
\end{equation}
where $\Phi^*$ is the adjoint of $\Phi$ with respect to the hermitian
metric on~$E$ which provides the Hitchin-Kobayashi correspondence 
(see Hitchin \cite{hit2}), and $f$ has finitely many critical values.

There is another important fact mentioned by Hitchin (see the original
version in Frankel~\cite{fra}, and its application to Higgs bundles in
Hitchin~\cite{hit2}): the critical points of $f$ are exactly the fixed
points of the circle action on~$\sM$.

If $(E, \Phi) = (E, e^{i\theta}\Phi)$ then $\Phi = 0$ with critical
value $c_0 = 0$. The corresponding critical submanifold is
$F_0 = f^{-1}(c_{0}) = f^{-1}(0) = \sN$, the moduli space of semistable
bundles. On the other hand, when $\Phi \neq 0$, there is a type of
algebraic structure for Higgs bundles introduced by
Simpson~\cite{sim1}: a \emph{variation of Hodge structure}, or
simply a \emph{VHS}, for a Higgs bundle $(E, \Phi)$ is a
decomposition:
\begin{equation}
E = \bigoplus_{j=1}^n E_j 
\word{such that} 
\Phi \: E_j \to E_{j+1} \ox K 
\word{for} j \leq n - 1
\word{and} \Phi(E_n) = 0.
\label{VHS} 
\end{equation} 

It has been proved by Simpson~\cite{sim2} that the fixed points of the
circle action on $\sM(r,d)$, and so, the critical points of $f$, are
these VHS, where the critical values $c_\la = f(E,\Phi)$ will depend 
on the degrees $d_j$ of the components $E_j \subset E$. By Morse theory, 
we can stratify $\sM$ in such a way that there is a nonzero critical 
submanifold $F_\la := f^{-1}(c_\la)$ for each nonzero critical value
$0 \neq c_\la = f(E,\Phi)$ where $(E,\Phi)$ represents a fixed point
of the circle action, or equivalently, a VHS. We then say that
$(E,\Phi)$ is an $(r_1,\dots,r_n)$-VHS, where $r_j = \rk(E_j)\ \forall j$.

\begin{Def} 
  A \emph{holomorphic triple} on $X$ is a triple
  $T = (E_1, E_2, \phi)$ consisting of two holomorphic vector bundles
  $E_1 \to X$ and $E_2 \to X$ and a homomorphism $\phi \: E_2 \to E_1$,
  i.e., an element $\phi \in H^0(\Hom(E_2,E_1))$.
\end{Def}

There are certain notions of $\sg$-degree:
  $$
  \deg_\sg(T) := \deg(E_1) + \deg(E_2) + \sg \cdot \rk(E_2),
  $$
and $\sg$-slope:
  $$
  \mu_\sg(T) := \frac{\deg_\sg(T)}{\rk(E_1) + \rk(E_2)}
  $$
which give rise to notions of $\sg$-stability of triples. The reader 
may consult the works of Bradlow and Garc\'ia-Prada~\cite{brgp}; 
Bradlow, Garc\'ia-Prada and Gothen~\cite{bgg1}; and Mu\~noz, Ortega 
and V\'azquez-Gallo~\cite{mov} for the details.

With this notions, one can construct:
 $$
 \sN_\sg = \sN_\sg(\rr,\dd) = \sN_\sg(r_1,r_2,d_1,d_2)
 $$
 the moduli space of $\sg$-polystable triples $T = (E_1,E_2,\phi)$ 
 such that $\rk(E_j) = r_j$ and $\deg(E_j) = d_j$,
 and 
 $$
 \sN^s_\sg = \sN^s_\sg(\rr,\dd)
 $$ 
 the moduli space of $\sg$-stable triples, where 
 $(\rr,\dd) = (r_1,r_2,d_1,d_2)$ is the type of the triple
 $T = (E_1,E_2,\phi)$.

We mention the moduli space $\sN_\sg(r_1,r_2,d_1,d_2)$ of $\sg$-stable
triples because they are closely related to some of the 
critical submanifolds~$F_\la$.

\begin{Def} 
  Fix a point $p \in X$, and let $\sO_X(p)$ be the associated
  line bundle to the divisor $p \in \Sym^1(X) = X$. A
  \emph{$k$-Higgs bundle} (or \emph{Higgs bundle with poles of
  order~$k$}) is a pair $(E,\Phi^k)$ where:
  $$
  E \xrightarrow{\ \Phi^k\ } E \ox K \ox \sO_X(kp) = E \ox K(kp)
  $$
  and where the morphism 
  $\Phi^k \in H^0\big(X, \End(E) \ox K(kp)\big)$ is what we call
  a \emph{Higgs field with poles of order~$k$}. The moduli space
  of $k$-Higgs bundles of rank~$r$ and degree~$d$ is denoted by
  $\sM^k(r,d)$. For simplicity, we will suppose that $\GCD(r,d) = 1$,
  and so, $\sM^k(r,d)$ will be smooth.
\end{Def}

There is an embedding
$$
i_k\: \sM^k(r,d) \to \sM^{k+1}(r,d) 
$$
$$
\big[(E,\Phi^k)\big] \longmapsto \big[(E,\Phi^k \ox s_p)\big]
$$
where $0 \neq s_p \in H^0(X, \sO_X(p))$ is a nonzero fixed section of~$\sO_X(p)$.

All the results mentioned for $\sM(r,d)$, hold also for $\sM^k(r,d)$.

\section{Generators for the Cohomology Ring} 
\label{sec:2}

According to {Hausel and Thaddeus \cite[(4.4)]{hath1}}, there is a universal 
family $(\mathbb{E}^k, \boldsymbol{\Phi}^k)$ over $X \times \sM^k$ where
$$
\left\{
  \begin{array}{r c l}
    \mathbb{E}^k & \to & X \times \sM^k(r,d)\\
    \boldsymbol{\Phi}^k & \in & H^0\big(\End(\mathbb{E}^k) \ox \pi_2^*(K(kp)) \big)
  \end{array}
\right.
$$
and from now on, we will refer $(\mathbb{E}^k,\boldsymbol{\Phi}^k)$ 
as a \emph{universal $k$-Higgs bundle}. Note that 
$(\mathbb{E}^k, \boldsymbol{\Phi}^k)$ satisfies the \emph{Universal Property}: 
in general, for any family $(\mathbb{F}^k, \boldsymbol{\Psi}^k)$ over $X \times M$, 
there is a morphism $\eta\: M \to \sM^k$ such that 
$
(\mathrm{Id}_X \times \eta)^*(\mathbb{E}^k, \boldsymbol{\Phi}^k) 
= (\mathbb{F}^k, \boldsymbol{\Psi}^k)
$. 
It means that, for $M = \sM^k$ whenever exists $(\mathbb{F}^k, \boldsymbol{\Psi}^k)$ 
such that
$$
(\mathbb{E}^k, \boldsymbol{\Phi}^k)_P \cong 
(\mathbb{F}^k, \boldsymbol{\Psi}^k)_P \quad 
\forall P = (E,\Phi^k) \in \sM^k(r,d),
$$ 
then, there exists a unique bundle morphism $\xi\: \mathbb{F}^k \to \mathbb{E}^k$ such that
\begin{align}
 \begin{xy}
  (0,25)*+{\mathbb{F}^k}="a";
  (30,25)*+{\mathbb{E}^k}="b";
  (15,0)*+{X \times \sM^k(r,d)}="c";
  {\ar@{-->}^{\exists ! \xi} "a";"b"};
  {\ar@{->}^{p_1} "b";"c"};
  {\ar@{->}_{p_2} "a";"c"};
 \end{xy}
\end{align}
commutes: $p_2 = p_1 \circ \xi$.

The universal bundle extends then to the following: if $(\mathbb{E}^k,\boldsymbol{\Phi}^k)$ and 
$(\mathbb{F}^k,\boldsymbol{\Psi}^k)$ are families of stable $k$-Higgs bundles parametrized by 
$\sM^k(r,d)$, such that $(\mathbb{E}^k, \boldsymbol{\Phi}^k)_P \cong (\mathbb{F}^k, \boldsymbol{\Psi}^k)_P$ 
for all $P = (E,\Phi^k) \in \sM^k(r,d)$, then there is a line bundle
$\mathcal{L} \to \sM^k(r,d)$ such that
$$
(\mathbb{E}^k,\boldsymbol{\Phi}^k) \cong (\mathbb{F}^k \ox \pi_2^*(\mathcal{L}),\boldsymbol{\Psi}^k \ox \pi_2),
$$
where $\pi_2\: X \times \sM^k(r,d) \to \sM^k(r,d)$ is the natural projection and the endomorphisms
satisfy $\boldsymbol{\Phi}^k \cong \boldsymbol{\Psi}^k \ox \pi_2(\sg_P) \cong \boldsymbol{\Psi}^k$,
where $\sg_P$ is a section of $X\times \sM^k \to \sM^k$. For more details, see 
{Hausel and Thaddeus \cite[(4.2)]{hath1}}.

\begin{Rmk}
 Do not confuse $\pi_2$ with $p_2$ (neither $\pi_1$ with $p_1$); $\pi_j$ are the natural projections 
 of the cartesian product, while $p_j$ are the bundle surjective maps:
 \begin{align}
 \begin{xy}
  (0,25)*+{\mathbb{E}^k}="a";
  (30,25)*+{\mathbb{F}^k}="b";
  (10,5)*+{}="c11";
  (20,5)*+{}="c12";
  (15,0)*+{X \times \sM^k(r,d)}="c";
  (10,-5)*+{}="c21";
  (20,-5)*+{}="c22";
  (0,-25)*+{X}="d";
  (30,-25)*+{\sM^k(r,d)}="e";
  {\ar@{->}_{p_1} "a";"c11"};
  {\ar@{->}^{p_2} "b";"c12"};
  {\ar@{->}_{\pi_1} "c21";"d"};
  {\ar@{->}^{\pi_2} "c22";"e"};
 \end{xy}
\end{align}
\end{Rmk}

If we consider the embedding $i_k\: \sM^k(r,d) \to \sM^{k+1}(r,d)$ for general rank, we get that:

\begin{Prop}
\label{UniversalHiggsBundleEmbedding}
  Let $(\mathbb{E}^k,\boldsymbol{\Phi}^k)$ be a universal Higgs bundle. Then:
 $$
 (\mathrm{Id}_X \times i_k)^*(\mathbb{E}^{k+1}) \cong \mathbb{E}^k.
 $$
\end{Prop}

\begin{proof}
Note that
$$\big(\mathbb{E}^k,\boldsymbol{\Phi}^k  \ox \pi_{1}^*(s_p)\big) \to X \times \sM^k$$
is a family of $(k+1)$-Higgs bundles on $X$, where $\pi_1\: X \times \sM^k \to X$ is the natural projection. 
So, by the universal property:
$$
\big(\mathbb{E}^k,\boldsymbol{\Phi}^k  \ox \pi_{1}^*(s_p)\big)
=
(\mathrm{Id}_X \times i_k)^*\big(\mathbb{E}^{k+1},\boldsymbol{\Phi}^{k+1}\big).
$$
\end{proof}

Consider
$$
\mathrm{Vect}(X) := \Big\{V\to X\: V\ \textmd{is a top. vector bundle} \Big\}\Big/ \cong
$$
the set of equivalence classes of topological vector bundles taken by isomorphism between 
them. Define the operation 
$$
[V]\oplus [W]\:=[V\oplus W]
$$
and consider the abelian semi-group $\big(\mathrm{Vect}(X),\oplus \big)$. Denote by
$$
K^0(X) = K\big(\mathrm{Vect}(X)\big) := \Big\{[V] - [W] \Big\}\Big/ \sim
$$
the abelian $K$-group of topological vector bundles on $X$, where
$$
[V] - [W] \sim 
[V\oplus U] - [W\oplus U]
$$
for every topological vector bundle $U\to X$.

Let $K^1(X)$ be the odd $K$-group of $X$ and let 
$$
K^*(X) = K^0(X)\oplus K^1(X)
$$
be the $K$-ring described by Atiyah \cite[Chapter II]{ati1}. 

In this case, $K^*(X)$ is torsion free since the Riemann surface $X$ is also a 
projective algebraic variety. Then, as a consequence of the K\"unneth Theorem 
(see for instance Atiyah \cite[Corollary~2.7.15.]{ati1} or \cite[Main Theorem]{ati2}), 
there is an isomorphism:

\begin{align}
 \begin{xy}
  (0,25)*+{\big(K^0(X)\ox K^0(\sM^k)\big) \oplus \big(K^1(X)\ox K^1(\sM^k)\big)}="a";
  (0,0)*+{K^0(X\times \sM^k)}="b";
  {\ar@{->}_{\cong} "a";"b"};
 \end{xy}
\end{align}

The reader may see Markman \cite{mar2} for the details. Furthermore, Markman \cite{mar2} 
chooses bases
$\{ x_1,\ ...,\ x_{2g} \} \subset K^1(X),$ and $\{ x_{2g+1},\ x_{2g+2} \} \subset K^0(X)$ 
to get a total basis
$$
\{ x_1,\ ...,\ x_{2g},\ x_{2g+1},\ x_{2g+2} \} \subset K^*(X) = K^0(X)\oplus K^1(X),
$$
and, since there is a universal bundle $\mathbb{E}^k \to X \times \sM^k$, 
we get the K\"unneth decomposition:
$$
[\mathbb{E}^k] = \sum_{j=0}^{2g}x_j  \ox e_j^k
$$
where $x_0\in K^0(X) = \mathrm{span}\{x_{2g+1},\ x_{2g+2}\}$,
$e_0^k\in K^0(\sM^k)$, $x_j\in K^1(X)$, and
$e_j^k\in K^1(\sM^k)$ for $j = 1,\ \dots,\ 2g$.
Finally, Markman \cite{mar2} considers the Chern classes $c_j(e_i^k)\in H^{2j}(\sM^k,\bZ)$ for 
$e_i^k\in K^*(\sM^k)$ and proves that:

\begin{Th}[{Markman \cite[Th.~3]{mar2}}]
\label{[Mar](2)}
 The cohomology ring $H^*\big(\sM^{k}(r,d),\bZ\big)$ is generated by the Chern classes of the 
 K\"unneth factors of the universal vector bundle.\QEDA
\end{Th}

\section{Preliminary Results} 
\label{sec:3}

Let $i_k\: \sM^k \hookrightarrow \sM^{k+1}$ be the embedding given by
the tensorization map of the $k$-Higgs field
$(E,\Phi^k) \longmapsto (E, \Phi^{k} \ox s_p)$, where $s_p$ 
is a fixed nonzero section of $L_p$. We want to prove that the map
$$
\pi_j(i_k) : \pi_j\big( \sM^{k}(r,d) \big) \to \pi_j\big( \sM^{k+1}(r,d) \big)
$$
stabilizes as $k\to \infty$. But first, we need to present some preliminary
results to conclude that.

\begin{Prop}
\label{embChernClasses}
  Let $i_k\: \sM^k \hookrightarrow \sM^{k+1}$ be the embedding above mentioned. 
  Consider the $K$-classes $e_i^k \in K(\sM^k)$. Then $i_k^*\big(c_j(e_i^{k+1}) \big) = c_j(e_i^{k})$.
\end{Prop}

\begin{proof}
By Proposition \ref{UniversalHiggsBundleEmbedding}, and by the naturality of the Chern classes:
$$
\sum_{j=0}^{2g} x_j  \ox e_j^k = [\mathbb{E}^k] 
= 
[(\mathrm{Id}_X \times i_k)^*(\mathbb{E}^{k+1})] 
= 
\sum_{j=0}^{2g} x_j  \ox i_k^*(e_j^{k+1})
$$
we have that
$i_k^*\big( e_i^{k+1} \big) = e_i^{k}$ and hence 
$i_k^*\big(c_j(e_i^{k+1}) \big) = c_j(e_i^{k})$. 
\end{proof}

\begin{Cor}
\label{SurjectiveCohomology}
 Let $i_k\: \sM^k \hookrightarrow \sM^{k+1}$ be the 
 embedding above mentioned. Then, the induced cohomology 
 homomorphism 
 $i_k^*\: H^*(\sM^{k+1},\bZ) \twoheadrightarrow H^*(\sM^{k},\bZ)$ 
 is surjective.
\end{Cor}

\begin{proof}
 The result is an immediate consequence 
 of Theorem \ref{[Mar](2)} and 
 Proposition \ref{embChernClasses}.
\end{proof}

\begin{Def}
  A \emph{gauge transformation} is an automorphism of $E$. Locally, a gauge transformation 
  $g\in Aut(E)$ is a $C^\infty(E)$-function with values in $GL_r(\bC)$. A gauge 
  transformation $g$ is called \emph{unitary} if $g$ preserves a hermitian inner product 
  on $E$. We will denote $\gG$ as the group of unitary gauge transformations. 
  {Atiyah and Bott \cite{atbo}} denote $\bar{\gG}$ as the quotient of $\gG$ 
  by its constant central $U(1)$-subgroup. We will follow this notation too. Moreover, denote 
  $B\gG$ and $B\bar{\gG}$ as the classifying spaces of $\gG$ and 
  $\bar{\gG}$, respectively.
\end{Def}

We get the fibration
$$
BU(1)\to B\gG \to B\bar{\gG}
$$
of classifying spaces, which splits actually as the product
$$
B\gG \cong BU(1)\times B\bar{\gG}.
$$
Then, the generators of $H^{*}(B\gG)$ give generators for $H^{*}(B\bar{\gG})$ and so,
$B\bar{\gG}$ is a free graded commutative algebra on those generators, since $B\gG$ is,
and consequently, $B\bar{\gG}$ is free of torsion. The reader may see 
{Atiyah and Bott \cite[Sec.~ 9.]{atbo}} 
and 
{Hausel \cite[Chap.~3]{hau}}
for the details.

Let $ \displaystyle \sM^\infty := \lim_{k\to \infty} \sM^k = \bigcup_{k=0}^\infty \sM^{k}$ 
be the direct limit of the spaces $\big\{ \sM^k(r,d) \big\}_{k=0}^\infty$. {Hausel and Thaddeus 
\cite{hath1}} prove that:

\begin{Th}[{Hausel and Thaddeus \cite[(9.7)]{hath1}}]
\label{[HaTh1](9.7)}
 The classifying space of $\bar{\gG}$ is homotopically equivalent to the 
 direct limit of the spaces $\sM^k(r,d)$:
 $$
 \displaystyle B\bar{\gG} \simeq \sM^\infty = \lim_{k\to \infty} \sM^k. 
 $$
 \QEDA
\end{Th}

\begin{assumption}
 Unless otherwise stated, from now on, we will assume that the rank is either $r = 2$ or $r = 3$.
\end{assumption}

\begin{Th}
\label{TorsionFreeRankTwoAndThree}
 $H^*\big(\sM^k(r,d)\big)$ is torsion free 
 for all $k$.
\end{Th}

\begin{proof}
The proof uses the following result of {Frankel \cite[Corollary~1]{fra}}:
$$
F_{\la}^k\quad \textmd{is torsion free}\quad \forall \la \Leftrightarrow \sM^k\quad \textmd{is torsion free}.
$$

In fact, the result of Frankel is more general. The specific case of moduli spaces of Higgs bundles holds
because the proper momentum Hitchin map $f(E,\Phi)$ described in (\ref{momentum_map_intro}) is a perfect 
Morse-Bott function, since we are taking $\mathrm{GCD}(r,d) = 1$.

In both cases, $r=2$ and $r=3$, the moduli space of stable vector bundles corresponds to the first critical 
submanifold:
$
F_0 = f^{-1}(c_{0}) = f^{-1}(0) = \sN
$,
which is indeed torsion free (see Atiyah and Bott \cite[Theorem~ 9.9.]{atbo}).

  \begin{enumerate}
  \item When $\rk(E)=2$, Hitchin notes that the nontrivial critical submanifolds, 
  or $(1,1)$-VHS, are of the form
  {\footnotesize
  $$
  F_{d_1}^k = 
  \Bigg\{(E,\Phi^k) = (E_1 \oplus E_2, \left( \begin{array}{c c}
					0 & 0\\ 
					\varphi_{21}^k & 0
					\end{array}
				       \right)
		      )\Bigg| 
  \begin{array}{c}
   \begin{array}{c c}
    \deg(E_1) = d_1, & \deg(E_2) = d_2,\\ 
    \rk(E_1) = 1, & \rk(E_2) = 1,\\
   \end{array}\\
   \varphi_{21}^k : E_1\to E_2  \ox K(kp)
  \end{array}
 \Bigg\}
  $$
  }
and $F_{d_1}^k$ is isomorphic to the moduli space of $\sg_H$-stable triples 
$\sN_{\sg_H}(1,1,\bar{d},d_1)$, where 
$\sg_H$ is giving by $\sg_H = \deg \big(K(kp)\big)=2g-2+k$ and 
$\bar{d}=d_2+2g-2+k$, by the map:
$$
(E_1  \ox E_2, \Phi^k) \mapsto (E_2  \ox K(kp), E_1, \varphi_{21}^k).
$$
Furthermore, by Hitchin \cite{hit2}, $\sN_{\sg_H}(1,1,\bar{d},d_1)$ is isomorphic 
to the cartesian product
$\sJ^{d_1}(X) \times \Sym^{\bar{d}-d_1}(X)$. Hence:
$$
F_{d_1}^k \cong \sJ^{d_1}(X) \times \Sym^{\bar{d}-d_1}(X)
$$
which, by Macdonald \cite[(12.3)]{mac}, is indeed torsion free.
  \item When $\rk(E)=3$, there are three kinds of nontrivial critical submanifolds:
      \subitem{2.1.} $(1,2)$-VHS of the form
    
    {\footnotesize
    $$
    F_{d_1}^k = \Bigg\{(E,\Phi^k) = (E_1 \oplus E_2, 
		       \left( \begin{array}{c c}
                        0 & 0\\ \varphi_{21}^k & 0
                       \end{array}
                       \right)
      )\Bigg| 
      \begin{array}{c}
      \begin{array}{c c}
      \deg(E_1) = d_1, & \deg(E_2) = d_2,\\ 
      \rk(E_1) = 1, & \rk(E_2) = 2,\\
      \end{array}\\
      \varphi_{21}^k : E_1\to E_2  \ox K(kp)
      \end{array}
      \Bigg\}.
     $$
     }
In this case, there are isomorphisms between the 
$(1,2)$-VHS and the moduli spaces of triples 
$F^k_{d_1} \cong \sN_{\sg_H(k)}(2,1,\tilde{d}_1,\tilde{d}_2)$, 
where $\tilde{d}_1=d_2+2(2g-2+k)$ and $\tilde{d}_2=d_1$, and 
where the isomorphism is giving by a map similar to the above 
mentioned. 

By {Mu\~noz, Ortega, V\'azquez-Gallo \cite[Theorem~4.8. and Lemma~6.1.]{mov}}, 
when working with 
$
\sN_{\sg}(2,1,\tilde{d}_1,\tilde{d}_2)
$,
they find that either the flip loci $S_{\sg_c}^+$ is the projectivization of a bundle of rank
$
r^{+} = \tilde{d}_1 - d_M - \tilde{d}_2
$
over 
$$
\sJ^{d_M}(X)\times \sJ^{\tilde{d}_2}(X)\times \Sym^{r^{+}}(X)
$$
where 
$\displaystyle
d_M = \frac{\sg_c + \tilde{d}_1 + \tilde{d}_2}{3}\in \bZ
$,
or the flip loci $S_{\sg_c}^-$ is the projectivization of a bundle of rank
$
r^{-} = 2d_M - \tilde{d}_1 +g - 1
$
over 
$$
\sJ^{d_M}(X)\times \sJ^{\tilde{d}_2}(X)\times \Sym^{r^{+}}(X)
$$
with
$\displaystyle
d_M \in \bZ
$
as above. Hence, by {Macdonald \cite[(12.3)]{mac}}, the flip loci $S_{\sg_c}^+$ and $S_{\sg_c}^-$ 
are free of torsion for $\sg_c \in I$. Therefore, $\sN_{\sg_H(k)}(2,1,\tilde{d}_1,\tilde{d}_2)$ is
also torsion free, and so is $F^k_{d_1}$.

The fact that $\sN_{\sg_H(k)}(2,1,\tilde{d}_1,\tilde{d}_2)$ is torsion free since the flip loci are, 
follows from the next lemma:

\begin{Lem}
\label{BlowUpLemma}
  Let $M$ be a complex manifold, and let $\Sigma \subset M$ be a complex submanifold. Let $\tilde{M}$ be the blow-up of $M$ along
  $\Sigma$. Then
  \begin{displaymath}
    H^*(\tilde{M},\bZ) \cong H^*(M,\bZ) \oplus H^{*+2}(\Sigma,\bZ) \oplus \dots \oplus H^{*+2n-2}(\Sigma,\bZ)
  \end{displaymath}
  where $n$ is the rank of $N_{\Sigma/M}$, the normal bundle of $\Sigma$ in $M$.
\end{Lem}

\begin{proof}(Lemma \ref{BlowUpLemma})\\
  Let $E=\mathbb{P}(N_{\Sigma/M})$ be the projectivized normal bundle of $\Sigma$ in $M$, sometimes called 
  \emph{exceptional divisor}. The result follows from the fact that the additive cohomology of the blow-up 
  $H^*(\tilde{M},\bZ)$, can be expressed as:
    \begin{displaymath}
      H^*(\tilde{M}) \cong \pi^*H^*(M) \oplus H^*(E) / \pi^*H^*(\Sigma)
    \end{displaymath}
  (see for instance {Griffiths and Harris \cite[Chapter~4.,Section~6.]{grha}}),
  and the fact that $H^*(E)$ is a free module over $H^*(\Sigma)$ via the injective map $\pi^*\colon H^*(\Sigma)\to H^*(E)$ with basis
    \begin{displaymath}
      1, c, \dots, c^{n-1},
    \end{displaymath}
  where $c \in H^2(E)$ is the first Chern class of the tautological line bundle along the fibres of the projective bundle $E \to \Sigma$\
  (see the general version at {Husemoller \cite[Chapter~17.,Theorem~2.5.]{hus}}).
\end{proof}

\subitem{2.2.} $(2,1)$-VHS of the form
    
    {\footnotesize
    $$
    F_{d_2}^k = \Bigg\{(E,\Phi^k) = (E_2 \oplus E_1, \left( \begin{array}{c c}
                        0 & 0\\ \varphi_{21}^k & 0
                       \end{array}
      \right)
      )\Bigg| 
      \begin{array}{c}
      \begin{array}{c c}
	\deg(E_2) = d_2, & \deg(E_1) = d_1,\\ 
	\rk(E_2) = 2, & \rk(E_1) = 1,\\
      \end{array}\\
      \varphi_{21}^k : E_2\to E_1  \ox K(kp)
      \end{array}
     \Bigg\}.
    $$
    }
By symmetry, similar results can be obtained using the isomorphisms 
between the $(2,1)$-VHS and the moduli spaces of triples:
$$
F^k_{d_2} \cong \sN_{\sg_H(k)}(1,2,\tilde{d}_1,\tilde{d}_2),
$$ 
and the dual isomorphisms 
$$
\sN_{\sg_H(k)}(2,1,\tilde{d}_1,\tilde{d}_2) 
\cong 
\sN_{\sg_H(k)}(1,2,-\tilde{d}_2,-\tilde{d}_1)
$$ 
between moduli spaces of triples.
    \subitem{2.3.} $(1,1,1)$-VHS of the form
 {\footnotesize
 $$
 F^k_{d_1 d_2 d_3} = 
\Bigg\{(E,\Phi^k) = (E_1 \oplus E_2 \oplus E_3, \left( \begin{array}{c c c}
                        0 & 0 & 0\\ \varphi^k_{21} & 0 & 0\\ 0 & \varphi^k_{32} & 0
                       \end{array}
 \right)
 )\Bigg| 
 \begin{array}{c}
  \deg(E_j) = d_j,\\ 
  rk(E_j) = 1,\\
  \varphi_{ij} : E_j \to E_i  \ox K
 \end{array}
\Bigg\}.
$$
}
Finally, we know that
$$
F^k_{d_1 d_2 d_3} \xrightarrow{\quad \cong \quad} 
\Sym^{m_1}(X)\times \Sym^{m_2}(X) \times \sJ^{d_3}(X)
$$
$$
(E,\Phi^k) \mapsto (\mathrm{div}(\varphi^k_{21}),\mathrm{div}(\varphi^k_{32}),E_3),
$$
where $m_i=d_{i+1}-d_i+\sg_H$, and so, by Macdonald \cite[(12.3)]{mac} there is nothing to worry about torsion.
 \end{enumerate}
\end{proof}

\begin{Cor}\label{[HaTh1](10.1)}
 $$
 \varprojlim H^*(\sM^k,\bZ)\cong 
 H^*(\sM^{\infty},\bZ)\cong 
 H^*(B\bar{\gG},\bZ).
 $$
 \QEDA
\end{Cor}

\begin{Cor} 
\label{CohomologyStabilization}
 For each $n\geq 0$ there is a $k_0$, depending on $n$, such that
 $$
 i_k^* : H^j(\sM^{k+1},\bZ) \xrightarrow{\quad \cong \quad} H^j(\sM^{k},\bZ)
 $$
 is an isomorphism for all $k \geq k_0$ and for all $j \leq n$.
 \QEDA
\end{Cor}

By the Universal Coefficient Theorem for Cohomology (see for instance 
{Hatcher \cite[Theorem~3.2. and Corollary~3.3.]{hat1}}), we get

\begin{Lem}
\label{HomologyStabilization}
 For each $n\geq 0$ there is a $k_0$, depending on $n$, such that 
 $$
 H_j(\sM^\infty,\sM^k;\bZ) = 0
 $$ 
 for all $k \geq k_0$ and for all $j \leq n$. 
\end{Lem}

\begin{proof}
 The embedding $i_k\: \sM^k(r,d) \to \sM^{k+1}(r,d)$ is injective, and by 
 Corollary \ref{SurjectiveCohomology}
 we know that $i_k^*\: H^j(\sM^{k},\bZ) \leftarrow H^j(\sM^{k+1},\bZ)$ is 
 surjective for all $k$. Hence, by the Universal Coefficient Theorem, we 
 get that the following diagram
 {\footnotesize
 \begin{align}
    \begin{xy}
  (10,10)*+{0}="b0";
  (40,10)*+{0}="c0";
  (70,10)*+{0}="d0";
  (-10,0)*+{0}="a1";
  (10,0)*+{\mathrm{Ext}\big(H_{j-1}(\sM^{k}),\bZ \big)}="b1";
  (40,0)*+{H^j(\sM^{k},\bZ)}="c1";
  (70,0)*+{\mathrm{Hom}\big(H_j(\sM^{k}),\bZ \big)}="d1";
  (90,0)*+{0}="e1";
  (-10,-15)*+{0}="a2";
  (10,-15)*+{\mathrm{Ext}\big(H_{j-1}(\sM^{k+1}),\bZ \big)}="b2";
  (40,-15)*+{H^j(\sM^{k+1},\bZ)}="c2";  
  (70,-15)*+{\mathrm{Hom}\big(H_j(\sM^{k+1}),\bZ \big)}="d2";
  (90,-15)*+{0}="e2";
  {\ar@{->} "a1";"b1"};
  {\ar@{->} "b1";"c1"};
  {\ar@{->} "c1";"d1"};
  {\ar@{->} "d1";"e1"};
  {\ar@{->} "a2";"b2"};
  {\ar@{->} "b2";"c2"};
  {\ar@{->} "c2";"d2"};
  {\ar@{->} "d2";"e2"};
  {\ar@{->}^{(i_{k*})^*} "b2";"b1"};
  {\ar@{->}^{i_k^*} "c2";"c1"};
  {\ar@{->}^{(i_{k*})^*} "d2";"d1"};
  {\ar@{->} "b1";"b0"};
  {\ar@{->} "c1";"c0"};
  {\ar@{->} "d1";"d0"};
    \end{xy}
 \end{align}
 }
 commutes. By Theorem~\ref{TorsionFreeRankTwoAndThree} $H^*(\sM^k,\bZ)$ 
 is torsion free, and so, by Corollary~\ref{CohomologyStabilization},
 for all $n\geq 0$, 
 there is $k_0$, depending on $n$, such that 
 $$
 H_j\big(\sM^{k}(r,d),\bZ\big) 
 \xrightarrow{\quad \cong \quad} 
 H_j\big(\sM^{k+1}(r,d),\bZ\big) 
 \xrightarrow{\quad \cong \quad} 
 H_j\big(\sM^\infty(r,d),\bZ\big)
 $$
 for all $k \geq k_0$ and 
 for all $j \leq n$. Hence
 $$
 H_j(\sM^\infty,\sM^k;\bZ) = 0
 $$
 for all $k \geq k_0$ and for all $j \leq n$.
\end{proof}

\begin{Prop}
\label{FundamentalGroups}
 For general rank $r$, denoting $\sM^k = \sM^k(r,d)$ for simplicity, and 
 $\sN = \sN(r,d)$ as the moduli of stable bundles, the following diagram 
 commutes
 \begin{align}
 \begin{xy}
  (0,25)*+{\pi_1(\sM^k)}="a";
  (0,0)*+{\pi_1(\sN)}="b";
  (30,25)*+{\pi_1(\sM^{k+1})}="c";
  (30,0)*+{\pi_1(\sN)}="d";
  {\ar@{->}^{\cong} "b";"a"};
  {\ar@{->}^{\cong} "a";"c"};
  {\ar@{->}_{\cong} "d";"c"};
  {\ar@{->}_{=} "b";"d"};
 \end{xy}
\end{align}
\end{Prop}

\begin{proof}
 It is an immediate consequence of the result proved by {Bradlow, Garc\'ia-Prada and 
 Gothen \cite[Proposition~3.2.]{bgg2}} using Morse theory.
\end{proof}

\begin{Prop}\label{FundamentalGroupsII}
 For all $k \in \bN$, there is an isomorphism between the fundamental group of $\sM^k$
 and the fundamental group of the direct limit of the spaces $\big\{ \sM^k(r,d) \big\}_{k=0}^\infty$:
 $$
 \pi_1(\sM^{k}) \xrightarrow{\quad \cong \quad} \pi_1(\sM^\infty).
 $$
\end{Prop}

\begin{proof}
 Using the generalization of Van Kampen's Theorem presented by {Fulton \cite{ful}}, and using the fact that 
 $\sM^{k} \hookrightarrow \sM^{k+1}$ are embeddings of \emph{Deformation Neighborhood Retracts} (DNR), \emph{i.e.} 
 every $\sM^{k}(r,d)$ is the image of a map defined on some open neighborhood of itself and homotopic to the identity 
 (see for instance {Hausel and Thaddeus \cite[(9.1)]{hath1}}), we can conclude that 
 $\displaystyle \pi_1\big(\lim_{k \to \infty} \sM^k \big) = \lim_{k \to \infty} \pi_1\big(\sM^k \big)$.
\end{proof}

\begin{Rmk}
 By {Atiyah and Bott \cite{atbo}} we have:
 $$
 \pi_1(\sN)\cong
 H_1(X,\bZ)\cong
 \bZ^{2g},
 $$
 and hence, by Proposition~\ref{FundamentalGroups} and Proposition~\ref{FundamentalGroupsII}:
 $$
 \pi_1(\sM^k)\cong
 \pi_1(\sM^{\infty})\cong
 \bZ^{2g}.
 $$
\end{Rmk}

We will need the following version of Hurewicz Theorem, presented by {Hatcher \cite[Theorem~4.37.]{hat1}} (see also {James \cite{jam}}). 
Hatcher first mentions that, in the relative case when $(X,A)$ is an $(n-1)$-connected pair of path-connected spaces, the kernel of 
the Hurewicz map
$$h : \pi_n(X,A) \to H_n(X,A;\bZ)$$
contains the elements of the form $[\gamma][f]-[f]$ for $[\gamma] \in \pi_1(A)$. Hatcher defines $\pi_n'(X,A)$ to be the quotient 
group of $\pi_n(X,A)$ obtained by factoring out the subgroup generated by the elements of the form $[\gamma][f]-[f]$, or the normal 
subgroup generated by such elements in the case $n=2$ when $\pi_2(X,A)$ may not be abelian, then $h$ induces a homomorphism 
$h'\: \pi'_n(X,A) \to H_n(X,A;\bZ)$. The general form of Hurewicz Theorem presented by Hatcher deals with this homomorphism:

\begin{Th}(Hurewicz Theorem)\\
\label{HurewiczThm}
 If $(X,A)$ is an $(n-1)$-connected pair of path-connected spaces, with $n\geq 2$ and $A\neq \emptyset$, then 
 $h'\: \pi'_n(X,A) \to H_n(X,A;\bZ)$ is an isomorphism and $H_j(X,A;\bZ)=0$ for $j \leq n-1$. \QEDA
\end{Th}

\begin{Def}
\begin{enumerate}[i.]
 \item 
 The \emph{determinant} of a vector bundle $E \to X$ of rank $r$ is a line bundle giving by the exterior power of the vector bundle. 
 It gives a natural map of the form:
 \begin{align*}
  \mathrm{det} : \sN &\xrightarrow{\hspace{1.5cm}} \sJ^d\\
  E &{\hspace{0.3cm}}\longmapsto{\hspace{0.3cm}} \mathrm{det}(E) = \bigwedge^{r}E
 \end{align*}
 where $\sN = \sN(r,d)$ is the moduli space of stable bundles $E \to X$ of rank $r$ and degree $d$, and $\sJ^{d}$
 is the Jacobian of $X$. Fixing a line bundle $\Lambda \to X$, $\Lambda \in \sJ^d$, the fibre 
 $\sN_\Lambda = \sN_\Lambda(r,d) := \mathrm{det}^{-1}(\Lambda)$ is the 
 moduli space of stable bundles \emph{with fixed determinant}.
 
 \item
 Together with the trace, the determinant allows us to define the map
 \begin{align*}
  \zeta : \sM^k(r,d) &\xrightarrow{\hspace{1.5cm}} \sJ^d\times H^0(X,K(kp))\\
  (E,\Phi) &{\hspace{0.5cm}}\longmapsto{\hspace{0.5cm}} \big(\mathrm{det}(E),\tr(\Phi)\big)
 \end{align*}
 and to consider the fibre 
 $\sM^k_\Lambda(r,d) := \zeta^{-1}(\Lambda,0)$ which is the moduli space of $k$-Higgs bundles 
 \emph{with fixed determinant and trace zero}.
\end{enumerate}
\end{Def}

There is an important result of {Atiyah and Bott \cite{atbo}} that is relevant to mention here:

\begin{Th}[{Atiyah and Bott \cite[(9.12.)]{atbo}}]
\label{[AtBo](9.12.)}
 The moduli space $\sN_\Lambda(r,d)$ of stable bundles of fixed determinant $\Lambda$, 
 with $\GCD(r,d) = 1$, is simply connected.\QEDA
\end{Th}

\begin{Rmk}
\begin{enumerate}[i.]
 \item 
 Some of the results mentioned for the moduli space $\sM^k(r,d)$ in this section 
 remain valid for the fixed determinant moduli space $\sM_{\Lambda}^k(r,d)$. For instance,
 Theorem~\ref{[HaTh1](10.1)} holds true also for fixed determinant: 
 $$
 \sM_{\Lambda}^{\infty}(r,d)\simeq B\bar{\gG}
 $$ 
 (see {Hausel and Thaddeus~\cite{hath1}}). Nevertheless, Corollary~\ref{SurjectiveCohomology}
 does not adapt in a straightforward way, as we shall see in subsection~\ref{ssec:4.2}.
 \item
 The moduli space $\sM_\Lambda^{k}(r,d)$ is simply connected because Proposition~\ref{FundamentalGroups} 
 holds also for fixed determinant $k$-Higgs bundles. So, $\pi_1(\sM_\Lambda^{k})$ acts trivially on 
 $
 \pi_n(\sM_\Lambda^\infty, \sM_\Lambda^{k})
 $.
\end{enumerate}
\end{Rmk}

\section{Main Results} 
\label{sec:4}

\subsection{General Results} 
\label{ssec:4.1}
Here, we will concern the moduli spaces $\sM^k(r,d)$ of $k$-Higgs bundles, where the 
results are true under the condition that $\pi_1(\sM^k)$ acts trivially on all the 
higher relative homotopy groups of the pair $\big(\sM^{\infty}, \sM^k\big)$. However, 
we do not know if this condition is true or not.

\begin{Lem}
\label{[Z-R](2.2.17.)}
 If $\pi_1(\sM^{k})$ acts trivially on $\pi_n(\sM^\infty, \sM^{k})$,
 then for all $n\geq 0$ exists $k_0$, depending on $n$, such that $\pi_j(\sM^\infty, \sM^{k}) = 0$ 
 for all $k \geq k_0$ and for all $j \leq n$.
\end{Lem}

\begin{proof}
 The proof proceeds by induction on 
 $m \in \bN$ for $2 \leq m \leq n$. The first induction step is trivial because
 $$\pi_1(\sN) = \pi_1(\sM) = \pi_1(\sM^{k}) = \pi_1(\sM^\infty)$$
 by Proposition \ref{FundamentalGroups}. For $m = 2$ we need $\pi_2(\sM^\infty,\sM^k)$ to be abelian. Consider 
 the sequence
 $$\pi_2(\sM^\infty) \to \pi_2(\sM^\infty,\sM^{k}) \to \pi_1(\sM^{k}) \to \pi_1(\sM^\infty) \to \pi_1(\sM^\infty,\sM^k) \to 0$$
 where $\pi_2(\sM^\infty) \twoheadrightarrow \pi_2(\sM^\infty,\sM^{k})$ is surjective, 
 $\pi_1(\sM^{k}) \xrightarrow{\ \cong\ } \pi_1(\sM^\infty)$ are isomorphic, and hence
 $\pi_1(\sM^\infty,\sM^k) = 0$. So, $\pi_2(\sM^\infty,\sM^k)$ is a quotient of the abelian group 
 $\pi_2(\sM^\infty)$, and so it is also abelian.\\
 Finally, suppose that the statement is true for all $j \leq m-1$ for $2 \leq m \leq n$. So, 
 $(\sM^\infty,\sM^{k})$ is $(m-1)$-connected, \emph{i.e.}
 $$\pi_j(\sM^\infty,\sM^{k}) = 0\quad \forall j \leq m-1.$$
 For $m\geq 2$, by Hurewicz Theorem \ref{HurewiczThm}, 
 $$h' : \pi'_m(\sM^\infty,\sM^{k}) \xrightarrow{\quad \cong \quad} H_m(\sM^\infty,\sM^{k};\bZ)$$ 
 is an isomorphism. By Lemma \ref{HomologyStabilization}, there is an integer $k_0$, depending on $m$,
 such that $H_m(\sM^\infty,\sM^{k};\bZ) = 0$ for all $k\geq k_0$. Hence, if
 $\pi_1(\sM^{k})$ acts trivially on  $\pi_n(\sM^\infty, \sM^{k})$ for all $n \in \bN$ and for all 
 $k \in \bN$, then
 $$\pi_m(\sM^\infty, \sM^{k}) = \pi'_m(\sM^\infty, \sM^{k}) = 0$$
 finishing the induction process.
\end{proof}

\begin{Cor}
\label{[Z-R](2.2.18.)}
 If $\pi_1(\sM^{k})$ acts trivially on $\pi_n(\sM^\infty, \sM^{k})$,
 then for all $n\geq 0$ exists $k_0$,depending on $n$, such that 
 $$
 \pi_j(\sM^{k}) \xrightarrow{\quad \cong \quad} \pi_j(\sM^\infty)
 $$
 for all $k \geq k_0$ and for all $j \leq n-1$.\QEDA
\end{Cor}  

\subsection{Fixed determinant case} 
\label{ssec:4.2}

The main goal here, is to avoid the hypothesis of the trivial action 
of the fundamental group on the relative homotopy group:
$\pi_1(\sM^{k}) \Circlearrowright \pi_n(\sM^\infty, \sM^{k})$. So, we
want to get the analogue of Lemma \ref{[Z-R](2.2.17.)} for $\sM_{\Lambda}^{k}$,
the moduli space of $k$-Higgs bundles with fixed determinant, since 
$\sM_{\Lambda}^{k}$ is simply connected. To do that, we will need the 
analogue of Corollary \ref{SurjectiveCohomology}, 
and then the analogue of Lemma \ref{HomologyStabilization} also for $\sM_{\Lambda}^k$.

The analogue of Corollary \ref{SurjectiveCohomology} for $\sM_{\Lambda}^k$ is not 
immediate. Note that the group of $r$-torsion points in the Jacobian:
$$
\Gamma=\Jac(r):=
\big\{ 
L \to X\ \textmd{line bundle}:\ L^r \cong \mathcal{O}_X
\big\}
$$
acts on $\sM_{\Lambda}^k(r,d)$ by tensorization:
$$
(E,\Phi^k) \mapsto (E\otimes L, \Phi^k \otimes \id_L).
$$

Hence, $\Gamma$ acts on $H^*(\sM_{\Lambda}^k,\bZ)$ for all $k$. This cohomology splits in a 
$\Gamma$-invariant part and in a complement which is called by Hausel and Thaddeus 
\cite{hath3} as the ``\emph{variant part}'':

\begin{equation}\label{jac_decomposition}
 H^*(\sM_{\Lambda}^k,\bZ) = H^*(\sM_{\Lambda}^k,\bZ)^\Gamma \oplus H^*(\sM_{\Lambda}^k,\bZ)^{var}.
\end{equation}

This decomposition appears in various cohomology calculations, see e.g., Hitchin \cite{hit2} 
for rank two, Gothen \cite{got} for rank three, Hausel \cite{hau} also for rank two, Bento \cite{ben} 
for the explicit calculations for rank two and rank three, and Hausel and Thaddeus \cite{hath3} 
for general rank.

The analogue of Corollary \ref{SurjectiveCohomology} for $\sM_{\Lambda}^k$ will be obtained for each 
of the pieces in the last direct sum (\ref{jac_decomposition}) separately:

\begin{itemize}
 \item For $H^*(\sM_{\Lambda}^k,\bZ)^\Gamma$:
 
 It follows from the corresponding result for $H^*(\sM^k,\bZ)$ 
 because there is a surjection $H^*(\sM^k,\bZ) \twoheadrightarrow H^*(\sM_{\Lambda}^k,\bZ)^\Gamma$.
 
Recall that, for general rank $r$, the moduli space of stable vector bundles corresponds 
to the first critical submanifold:
$
F_0 = f^{-1}(c_{0}) = f^{-1}(0) = \sN(r,d)
$. 
The group $\Gamma$ acts trivially on $H^* (\sN,\bZ)$, and there is a surjection
$$
H^*(\sN,\bZ)\twoheadrightarrow H^*(\sN_{\Lambda},\bZ).
$$
The reader may see Atiyah and Bott \cite[Prop.~ 9.7.]{atbo} for the details.

 For the rank $r = 2$ case, a nontrivial critical submanifold of $\sM_{\Lambda}^k(2,1)$, is 
 a so-called $(1,1)$-VHS:
{\footnotesize
 $$
  F_{d_1}^k(\Lambda) = 
  \Bigg\{(E,\Phi^k) = (E_1 \oplus E_2, 
  \left(
    \begin{smallmatrix}
	0 & 0\\ 
	\varphi_{21}^k & 0
    \end{smallmatrix}
  \right)
  )
  \Bigg| 
  \begin{array}{c}
   \begin{array}{c c}
    \deg(E_j) = d_j, & \rk(E_j) = 1,\\ 
   \end{array}\\
   \varphi_{21}^k : E_1\to E_2  \ox K(kp),\\
   E_1 E_2 = \Lambda
  \end{array}
 \Bigg\},
 $$
 }
 which is a $2^{2g}$-covering with covering group the $2$-torsion points in the Jacobian 
 $\Gamma \cong (\bZ_2)^{2g}$. Hence, the results of Betti numbers presented by Bento 
 \cite[Prop.~2.2.3.]{ben} let us conclude the following:
 
 \begin{Prop}
  The cohomology map 
  $$
  H^*\big(\Sym^m(X),\bZ \big) \to H^*\big(F^k_{d_1}(\Lambda),\bZ \big)
  $$
  induced by the $\Gamma$-covering
  $
  F_{d_1}^k(\Lambda) \to \Sym^m(X)
  $
  where $m = d_2 -d_1 + 2g - 2 + k$, is injective, and its image is the $\Gamma$-invariant subgroup 
  $H^*\big(F_{d_1}^k(\Lambda),\bZ \big)^{\Gamma}$.\QEDA
 \end{Prop}
 
 \begin{Cor}\label{SurjGammaInv2}
  There exists a surjection
  $$
  H^*\big(\sM^k(2,1),\bZ \big)\twoheadrightarrow H^*\big(\sM_{\Lambda}^k(2,1),\bZ \big)^{\Gamma}.
  $$
  \QEDA
 \end{Cor}

 When $r = 3$, the group of $3$-torsion points in the Jacobian looks like $\Gamma \cong (\bZ_3)^{2g}$, and 
 the nontrivial critical submanifolds of $\sM^k_{\Lambda}(3,d)$ are VHS either of type $(1,2),\ (2,1)$ or 
 $(1,1,1)$, where the cohomology of the $(1,2)$ and $(2,1)$ VHS is invariant under the action of $\Gamma$, 
 and the $(1,1,1)$-VHS is a $3^{2g}$-covering of $\Sym^{m_1}(X)\times \Sym^{m_2}(X)$ with covering group 
 $\Gamma \cong (\bZ_3)^{2g}$. Hence:
 
 \begin{Prop}
  $$
  H^*\big(F_{d_1}^k(\Lambda),\bZ \big)
  =
  H^*\big(F_{d_1}^k(\Lambda),\bZ \big)^{\Gamma}\
  \textmd{and}\quad
  H^*\big(F_{d_2}^k(\Lambda),\bZ \big)
  =
  H^*\big(F_{d_2}^k(\Lambda),\bZ \big)^{\Gamma}
  $$
  where
  {\footnotesize
  $$
  F_{d_1}^k(\Lambda)
  = 
  \Bigg\{(E,\Phi^k) = (E_1 \oplus E_2, 
		  \left( \begin{array}{c c}
			    0 & 0\\ 
			    \varphi_{21}^k & 0
			 \end{array}
		 \right)
  )\Bigg| 
  \begin{array}{c}
    \begin{array}{c c}
      \deg(E_1) = d_1, & \deg(E_2) = d_2,\\ 
      \rk(E_1) = 1, & \rk(E_2) = 2,\\
    \end{array}\\
    \varphi_{21}^k : E_1\to E_2  \ox K(kp)\\
    E_1 E_2 = \Lambda
  \end{array}
  \Bigg\}
  $$
  }
  and 
  {\footnotesize
  $$
  F_{d_2}^k(\Lambda)
  = 
  \Bigg\{
    (E,\Phi^k) = (E_2 \oplus E_1, \left( 
				    \begin{array}{c c}
				    0 & 0\\ 
				    \varphi_{21}^k & 0
				    \end{array}	
				  \right)
  )\Bigg| 
   \begin{array}{c}
   \begin{array}{c c}
    \deg(E_2) = d_2, & \deg(E_1) = d_1,\\ 
    \rk(E_2) = 2, & \rk(E_1) = 1,\\
   \end{array}\\
   \varphi_{21}^k : E_2\to E_1  \ox K(kp)\\
   E_2 E_1 = \Lambda
   \end{array}
  \Bigg\}
  $$
  }
  are the $(1,2)$ and $(2,1)$-VHS of $\sM^k_ {\Lambda}(3,d)$ respectively, with 
  $$
  \frac{d}{3}\leq d_1\leq  \frac{d}{3} + \frac{2g-2+k}{2}\
  and\quad  
  \frac{2d}{3}\leq d_2\leq \frac{2d}{3} + \frac{2g-2+k}{2}.
  $$
  Furthermore:
  $$
  H^*\big(F_{m_1 m_2}^k(\Lambda),\bZ \big) = H^*\big(F_{m_1 m_2}^k(\Lambda),\bZ \big)^{\Gamma}\oplus H^*\big(F_{m_1 m_2}^k(\Lambda),\bZ \big)^{var}
  $$
  and the cohomology map 
  $$
  H^*\big(\Sym^{m_1}(X)\times \Sym^{m_2}(X),\bZ \big) \to H^*\big(F^k_{m_1 m_2}(\Lambda),\bZ \big)
  $$
  induced by the $\Gamma$-covering
  $
  F_{m_1 m_2}^k(\Lambda) \to \Sym^{m_1}(X)\times \Sym^{m_2}(X)
  $
  where
  $
  F^k_{m_1 m_2}(\Lambda) = 
  $
  {\footnotesize
  $$
    \Bigg\{(E,\Phi^k) = 
    (E_1 \oplus E_2 \oplus E_3, \left( 
				  \begin{array}{c c c}
				    0 & 0 & 0\\ 
				    \varphi^k_{21} & 0 & 0\\ 
				    0 & \varphi^k_{32} & 0
				  \end{array}
				\right) )\Bigg| 
  \begin{array}{c}
  \deg(E_j) = d_j,\ rk(E_j) = 1,\\
  \varphi_{ij} : E_j \to E_i  \ox K(kp)\\
  E_1 E_2 E_3 = \Lambda
  \end{array}
  \Bigg\},
  $$
}
  is the $(1,1,1)$-VHS of $\sM^k_ {\Lambda}(3,d)$ with $m_j = d_{j+1} -d_j + 2g - 2 + k$, 
  is injective, and its image is the $\Gamma$-invariant subgroup $H^*\big(F_{m_1 m_2}^k(\Lambda) \big)^{\Gamma}$.\QEDA
 \end{Prop}
 
 \begin{Cor}\label{SurjGammaInv3}
  There exists a surjection
  $$
  H^*\big(\sM^k(3,d),\bZ \big)\twoheadrightarrow H^*\big(\sM_{\Lambda}^k(3,d),\bZ \big)^{\Gamma}.
  $$
  \QEDA
 \end{Cor}

 The reader may see Bento \cite{ben}, Gothen \cite{got} and also Hausel and Thaddeus \cite{hath3}, 
 for details. Using the results above, we get:
 
\begin{Lem}\label{CohoStabFixedDetGammaInv}
 The induced cohomology homomorphism restricted to the $\Gamma$-invariant cohomology 
 of the moduli spaces of $k$-Higgs bundles with fixed determinant $\Lambda$
 $$
 i_k^*\: H^*(\sM^{k+1}_{\Lambda}(r,d),\bZ)^{\Gamma} 
 \twoheadrightarrow 
 H^*(\sM^{k}_{\Lambda}(r,d),\bZ)^{\Gamma}
 $$ 
 is surjective.
\end{Lem}

\begin{proof}
 It is enough to note that the following diagram
 \begin{align}
 \begin{xy}
  (0,25)*+{H^*(\sM^{k+1},\bZ)}="a";
  (0,0)*+{H^*(\sM^{k+1}_{\Lambda},\bZ)^{\Gamma}}="b";
  (40,25)*+{H^*(\sM^{k},\bZ)}="c";
  (40,0)*+{H^*(\sM^{k}_{\Lambda},\bZ)^{\Gamma}}="d";
  {\ar@{->>} "a";"b"};
  {\ar@{->>}^{i_k^{*}} "a";"c"};
  {\ar@{->>} "c";"d"};
  {\ar@{-->}_{i_k^{*}} "b";"d"};
 \end{xy}
\end{align}
 commutes, where the top arrow is surjective by Corollary \ref{SurjectiveCohomology},
 and the descending arrows are surjective because of Corollary \ref{SurjGammaInv2} 
 and Corollary \ref{SurjGammaInv3}.
\end{proof}
 
 \item For $H^*(\sM_{\Lambda},\bZ)^{var}$:
 
 First, note that with fixed determinant $\Lambda$ the critical submanifolds of 
 type $(1,1)$ and $(1,1,1)$ are $r^{2g}$-coverings with covering group 
 $\Gamma \cong (\bZ_r)^{2g}$, with $r = 2$ or $r = 3$ (see Bento \cite{ben} Prop.~2.2.1. and Lemma~ 2.4.4.). 
 Furthermore, when $r = 3$ the cohomology of $(1,2)$ and $(2,1)$ critical submanifolds is $\Gamma$-invariant. Then, only the 
 cohomology of $(1,1)$-VHS and $(1,1,1)$-VHS split in the $\Gamma$-invariant part and the \emph{variant} complement, for rank 
 $r = 2$ and $r = 3$, respectively. Hence:
 $$
 H^{*}\big(\sM^k_{\Lambda}(2,1),\bZ \big)^{var}
 =
 \bigoplus_{d_1 > \frac{1}{2}}^{\frac{1 + d_k}{2}}H^{*}\big(F^k_{d_1}(\Lambda)\big)^{var}\
 \textmd{and}
 $$
 $$
 H^{*}\big(\sM^k_{\Lambda}(3,d),\bZ \big)^{var}
 =
 \bigoplus_{(m_1,m_2)\in \Omega_{d_k}}H^{*}\big(F^k_{m_1 m_2}(\Lambda),\bZ \big)^{var}
 $$
 where 
 $d_k = \deg\big(K\ox \sO_X(kp)\big) = \deg\big(K(kp)\big) = 2g - 2 + k$, 
 $\frac{1}{2} < d_1 < \frac{1 + d_k}{2}$ according to Hitchin \cite{hit2} 
 for $(1,1)$-VHS in rank two, and $(m_1,m_2) \in \Omega_{d_k}$ where 
 $
 M_j:=E_j^{*}E_{j+1}K(kp)
 $, 
 $m_j:=\mathrm{deg}(M_j)=d_{j+1}-d_j+d_k$, and the set of indexes
 $$
 \Omega_{d_k} = 
 \Bigg\{ (m_1,m_2) \in \mathbb{N} \times \mathbb{N} \Bigg| 
 \begin{array}{c}
  2m_1 + m_2 < 3d_k\\
  m_1 + 2m_2 < 3d_k\\
  m_1 + 2m_2 \equiv d(3)
 \end{array}
 \Bigg\}
 $$
 for $(1,1,1)$-VHS in rank three is described by Bento \cite[Prop.~2.3.9.]{ben}, 
 Gothen \cite[Sec.~3.]{got}, Gothen and Z\'u\~niga-Rojas \cite[Subsec.~5.1]{gzr}, 
 among others.
 
 There are some results appearing in the work of Bento \cite{ben} 
 (Lemma~2.2.4. and Prop.~2.2.5 for $\sM^k_{\Lambda}(2,1)$ and $F_{d_1}^k(\Lambda)$ its $(1,1)$-VHS, and 
 Lemma~2.4.4. and Prop.~2.4.5. for $\sM^k_{\Lambda}(3,d)$ and $F_{m_1 m_2}^k(\Lambda)$ its $(1,1,1)$-VHS)
 where Bento works with Hitchin pairs twisted by a general line bundle $L$ of degree $\deg(L) = d_L$, 
 and the following results below correspond to the particular case of $k$-Higgs bundles with 
 $L = K(kp)$, and hence $d_L = d_k = 2g - 2 + k$:
 
 \begin{Lem}
 \label{Rk2VarCompl}
  Let $F_{d_1}^k(\Lambda)$ be a $(1,1)$-VHS of $\sM^k_{\Lambda}(2,1)$ and let 
  $m = d_2 - d_1 + 2g - 2 + k$. Then
  $$
  H^j(F_{d_1}^k(\Lambda),\bZ)^{var} \neq 0 \Longleftrightarrow j = m.
  $$
 \end{Lem}
 
 \begin{proof}
  See Bento \cite[Prop.~ 2.2.4.]{ben}.
 \end{proof}
 
 \begin{Lem}
 \label{Rk3VarCompl}
  Let $F_{m_1 m_2}^k(\Lambda)$ be a $(1,1,1)$-VHS of $\sM^k_{\Lambda}(3,d)$. Then
  $$
  H^i\big(F_{m_1 m_2}^k(\Lambda),\bZ \big)^{var} \neq 0 \Longleftrightarrow i = m_1 + m_2,
  $$
  where $m_j = d_{j+1} - d_j + d_k$.
 \end{Lem}
 
 \begin{proof}
  See Bento \cite[Prop.~ 2.4.4.]{ben}.
 \end{proof}
 
 Then, in both cases, when $r = 2$ and when $r = 3$, the cohomology 
 groups with integer coefficients are torsion free:
 
 \begin{itemize}
  \item 
  If $r = 2$, we have just one nonzero component, $H^m(F_{d_1}^k(\Lambda),\bZ)^{var}$ which is the sum of $2^{2g}$ 
  copies of $H^m(\Sym^m(X),\bZ)$, since $F_{d_1}^k(\Lambda)\to \Sym^m(X)$ is a $(\bZ_2)^{2g}$-covering.
  \item
  Similarly, if $r = 3$, the nonzero component is $H^{m_1+m_2}\big(F_{m_1 m_2}^k(\Lambda),\bZ \big)^{var}$
  which is the sum of $3^{2g}$ copies of $H^{m_1+m_2}\big(\Sym^{m_1}(X)\times \Sym^{m_2}(X),\bZ\big)$, since 
  $$
  F_{m_1 m_2}^k(\Lambda)\to \Sym^{m_1}(X)\times \Sym^{m_2}(X)
  $$ 
  is a $(\bZ_3)^{2g}$-covering.
 \end{itemize}
 They are torsion free by {Macdonald \cite[(12.3)]{mac}}. The reader may consult {Bento \cite[Chap.~2]{ben}} 
 for the details. Hence, we get
 
\begin{Lem}\label{CohoStabFixedDetGammaVar}
 Let $i_k\: \sM^k_{\Lambda} \hookrightarrow \sM^{k+1}_{\Lambda}$ be the embedding given by 
 the tensorization map
 $
 (E,\Phi^k) \mapsto (E, \Phi^{k} \ox s_p)
 $ as above mentioned. Then, the induced cohomology homomorphism 
 $$
 i_k^*\: H^*(\sM^{k+1}_{\Lambda},\bZ)^{var} \twoheadrightarrow H^*(\sM^{k}_{\Lambda},\bZ)^{var}
 $$ 
 is surjective, restricted to the variant complement.\QEDA
\end{Lem}

 This latter method only works wtih rank $r = 2$ or $r = 3$, but not in general. The difficulty in calculating 
 $H^*(\sM_{\Lambda},\bZ)^{var}$ for general rank is explained also on Hausel and Thaddeus \cite{hath3}.
\end{itemize}

Finally, we may conclude the following:

\begin{Cor}
\label{SurjectiveCohomologyFixedDeterminant}
 Let $i_k\: \sM^k_{\Lambda} \hookrightarrow \sM^{k+1}_{\Lambda}$ be the embedding above mentioned. Then, the 
 induced cohomology homomorphism
 $$
 i_k^*\: H^*(\sM^{k+1}_{\Lambda},\bZ) \twoheadrightarrow H^*(\sM^{k}_{\Lambda},\bZ)
 $$ 
 is surjective.
\end{Cor}

\begin{proof}
  It is enough to see that the cohomology of $\sM^k_{\Lambda}$ splits in the $\Gamma$-invariant part and the 
  \emph{variant} complement:
  $$
  H^*(\sM^k_{\Lambda},\bZ)
  =
  H^*(\sM^k_{\Lambda},\bZ)^{\Gamma}
  \oplus
  H^*(\sM^k_{\Lambda},\bZ)^{var}
  $$
  and so, the result follows from Lemma \ref{CohoStabFixedDetGammaInv} and Lemma \ref{CohoStabFixedDetGammaVar}.
\end{proof}

\begin{Lem}
\label{HomologyStabilizationFixedDet}
 For all $n$ exists $k_0$, depending on $n$, such that 
 $$
 H_j(\sM^\infty_{\Lambda},\sM^k_{\Lambda};\bZ) = 0
 $$ 
 for all $k \geq k_0$ and for all $j \leq n$. \QEDA
\end{Lem}

\begin{Th}
\label{[Z-R]MainResult1FixedDet}
 For all $n$ exists $k_0$, depending on $n$, such that 
 $$
 \pi_j(\sM_{\Lambda}^\infty, \sM_{\Lambda}^{k}) = 0
 $$ 
 for all $k \geq k_0$ and for all $j \leq n$.
\end{Th}

\begin{proof}
 The proof is quite similar to the proof of Lemma \ref{[Z-R](2.2.17.)}, 
 using now Corollary \ref{SurjectiveCohomologyFixedDeterminant}
 and Lemma \ref{HomologyStabilizationFixedDet}, and so, we have 
 a new advantage: $\sM^k_{\Lambda}$ is simply connected, hence 
 the action 
 $
 \pi_1(\sM_\Lambda^{k})\Circlearrowright \pi_n(\sM_\Lambda^\infty, \sM_\Lambda^{k})
 $ 
 is  trivial. 
\end{proof}

\begin{Cor}
\label{[Z-R]MainResult2FixedDet}
 For all $n$ exists $k_0$, depending on $n$, such that 
 $$
 \pi_j(\sM_{\Lambda}^{k}) \xrightarrow{\quad \cong \quad} \pi_j(\sM_{\Lambda}^\infty)
 $$
 for all $k \geq k_0$ and for all $j \leq n-1$.\QEDA
\end{Cor}  

\section*{Acknowledgement}
\addcontentsline{toc}{section}{Acknowledgement}

\quad Part of this paper is partially based on my Ph.D. thesis \cite{z-r} and I would like to thank 
my supervisor Peter B. Gothen for introducing me to the subject of Higgs bundles, and 
for all his patience during our illuminating discussions. Mange tak!

I would like to thank Andr\'e Gama Oliveira for the advice of working with fixed determinant,
and so, considering the trivial action 
$\pi_1(\sM_{\Lambda}^{k}) \Circlearrowright \pi_n(\sM_{\Lambda}^\infty, \sM_{\Lambda}^{k})$. 
Muito obrigado!

I am grateful to the referee for a very careful reading of my manuscript, specially for
pointing out the necessary conditions on the indexes of the critical submanifolds.

I also would like to thank Joseph C. V\'arilly for his time, listening and reading my results. 
Thanks for every single advice. Go raibh maith agat!

I benefited from the 
VBAC-Conference 2014, held in Freie Universit\"at Berlin, 
and the 
VBAC-Conference 2016, held in Centre Interfacultaire Bernoulli, 
both organized by the
Vector Bundles and Algebraic Curves research group of the European Research Training Network EAGER.

Finally, I acknowledge the financial support from CIMM, Centro de Investigaciones Matem\'aticas 
y Metamatem\'aticas, here in Costa Rica nowadays as a researcher, 
through the Project 820-B5-202; and also the financial support from 
FEDER through Programa Operacional Factores de Competitividade-COMPETE, and
FCT, Funda\c{c}\~ao para a Ci\^encia e a Tecnologia, there in Portugal,  
through PTDC/MAT-GEO/0675/2012 and PEst-C/MAT/UI0144/2013 with grant reference 
SFRH/BD/51174/2010, when I was working on my Ph.D. thesis.
\textexclamdown Pura Vida! Muito obrigado!

\appendix 

\renewcommand*{\refname}{}

\begin{flushright}
  \emph{Ronald A. Z\'u\~niga-Rojas}\\
  \small Centro de Investigaciones Matem\'aticas\\
  \small y Metamatem\'aticas CIMM\\
  \small Universidad de Costa Rica UCR\\
  \small San Jos\'e 11501, Costa Rica\\
  \small e-mail: \texttt{ronald.zunigarojas@ucr.ac.cr}
\end{flushright}

\end{document}